\documentclass[12pt]{amsart}

\usepackage{amscd}
\usepackage{amsopn}
\usepackage{cases}
\newcommand{\Mo}{(M,\omega )}

\newtheorem{lemma}{Lemma}
\newtheorem{theorem}{Theorem}
\newtheorem{proposition}{Proposition}
\newtheorem{corollary}{Corollary}
\theoremstyle{remark}
\newtheorem{remark}{Remark}
\newtheorem{example}{Example}

\begin{document}

\title[New constructions of symplectically fat fiber bundles]{New constructions of symplectically fat fiber bundles}

\author[Boche\'nski]{Maciej Boche\'nski}

\address{
Department of Mathematics and Computer Science\\
University of Warmia and Mazury\\
S\l\/oneczna 54\\
10-710 Olsztyn\\
Poland}

\email{mabo@matman.uwm.edu.pl}

\author[Szczepkowska]{Anna Szczepkowska}

\address{
Department of Mathematics and Computer Science\\
University of Warmia and Mazury\\
S\l\/oneczna 54\\
10-710 Olsztyn\\
Poland}

\email{anna.szczepkowska@matman.uwm.edu.pl}

\author[Tralle]{\mbox{Aleksy Tralle}}

\address{
Department of Mathematics and Computer Science\\
University of Warmia and Mazury\\
S\l\/oneczna 54\\
10-710 Olsztyn\\
Poland}

\email{tralle@matman.uwm.edu.pl}

\author[Woike]{\mbox{Artur Woike}}

\address{
Department of Mathematics and Computer Science\\
University of Warmia and Mazury\\
S\l\/oneczna 54\\
10-710 Olsztyn\\
Poland}

\email{awoike@matman.uwm.edu.pl}

\begin{abstract}
This work is devoted to new constructions of symplectically fat fiber bundles. The latter are constructed in two ways: using the Kirwan map and expressing the fatness condition in terms of the isotropy representation related to the $G$-structure over some homogeneous spaces.
\end{abstract}

\maketitle

\section{Introduction}

It is well known that basically there are two general ways to endow the total space of a fiber bundle
$$F\rightarrow M\rightarrow B$$
with a fiberwise symplectic form, that is, a symplectic form on $M$ which restricts to a symplectic form on the fiber. The first one is given by Thurston's theorem \cite{ms} (Theorem 6.3).

\begin{theorem}[Thurston] Let there be given a fiber bundle over a compact symplectic base $B$ and a symplectic fiber $(F,\sigma)$. Assume that 
\begin{enumerate}
\item the structure group of the bundle reduces to the group of symplectomorphisms of the fiber,
\item there exists a cohomology class $a\in H^2(M)$ which restricts to the cohomology class $[\sigma]$ on the fiber.
\end{enumerate}
Under these assumptions $M$ admits a fiberwise symplectic form.
\end{theorem}
\noindent The second way due to Sternberg, Weinstein and Lerman \cite{ster},\cite{wein},\cite{ler} is described as follows.

Let  $G\rightarrow P\rightarrow B$ 
be a principal bundle with a connection.
Let $\theta$ and $\Theta$ be the connection one-form and
the curvature form of the connection, respectively.  Both forms have
values in the Lie algebra $\mathfrak{g}$ of the group $G$. 
Denote the pairing between $\mathfrak{g}$ and
its dual $\mathfrak{g}^*$ by $\langle\,,\rangle$.  By definition, a vector
$u\in\mathfrak{g}^*$ is {\em  fat}, if the two--form
$$
(X,Y)\rightarrow \langle\Theta(X,Y),u\rangle
$$
is non-degenerate for all {\it horizontal} vectors $X,Y$. Note
that if a connection admits a fat vector $u$ then  the whole coadjoint orbit of $u$ consists of fat vectors. 

The following result is due to Sternberg and Weinstein.

\begin{theorem}[Sternberg-Weinstein]\label{thm:s-w} 
Let $\Mo$ be a a symplectic manifold with a Hamiltonian action of a
Lie group $G$ and a moment map $\mu: M\rightarrow\mathfrak{g}^*$.
Let $G\rightarrow P\rightarrow B$ be a principal bundle. 
If there exists a connection in the principal bundle $P$ such that all
vectors in $\mu(M)\subset\mathfrak{g}^*$ are fat, then the 
 total space of the associated bundle
$$
M\rightarrow P\times_G M\rightarrow B
$$
admits a symplectic form which restricts to the given symplectic form on the fiber.
\label{sw}
\end{theorem}

\noindent Throughout this article  we will call the fiber bundles satisfying the assumptions of Theorem \ref{thm:s-w}  {\it symplectically fat}. The primary purpose of this article is to construct new examples of such bundles.The idea of constructing fiberwise symplectic structures on associated bundles using fatness of $\mu(M)$ is  due to Weinstein \cite{wein} (Theorem 3.2 in the cited work). 

There are several reasons to do this. The most important one comes from symplectic topology (cf. \cite{ms}, Chapter 6). Symplectically fat fiber bundles yield families  of symplectic manifolds with various prescribed properties serving as important testing examples.

The second reason of our interest in symplectically fat bundles comes from the book by Guillemin, Lerman and Sternberg \cite{gls}. In this book, the authors show how to apply fat bundle constructions in representation theory. Basically, the authors use the particular case of fat bundles which are coadjoint orbits fibered over coadjoint orbits. However, an analysis of the general case may yield new applications.
We also want to mention an important article \cite{fp} which is related to symplectic fatness. The authors constructed there some non-K\"ahler manifolds with trivial canonical bundle by using a special case of a theorem of Reznikov  on the existence of  symplectic structures on some twistor bundles \cite{rez}. The latter result turned out to be a consequence of the fact that these  bundles are symplectically fat   \cite{ktw}.

Although the method of fat bundles seems interesting and useful, it is extremely difficult to even find  examples satisfying the fatness condition. Surprisingly, the only known examples of it are the following classes of bundles (see \cite{ktw}):
\begin{enumerate}
\item Let $G$ be a semisimple Lie group and $H\subset G$ compact subgroup of maximal rank. Bundles of the form
$$H/K\rightarrow G/K=G\times_H(H/K)\rightarrow G/H$$
where $K=Z_G(T)\subset H$ for some torus in $G$ satisfy the fatness condition.
\item Another example are twistor bundles of the form
$$SO(2n)/U(n)\rightarrow \mathcal{T}(M)\rightarrow M,$$
where $(M^{2n},g)$ is an even-dimensional Riemannian manifold with $\frac{3}{2n+1}$-pinched sectional curvature $K_{g}$ ($K_{g}$ satisfies $1-\frac{3}{2n+1} \leq |K_{g}| \leq 1$).
\item Also locally homogeneous complex manifolds $\Gamma\setminus G/V$  fibered over locally symmetric Riemannian manifolds as follows:
$$K/V\rightarrow \Gamma\setminus G/V\rightarrow \Gamma\setminus G/K,$$
where $G$ is a semisimple Lie group of non-compact type, $\Gamma$ is a uniform lattice in $G$, $K$ a maximal compact subgroup in $G$ and $V=Z_G(T)\subset K$ for some torus $T$ in $G$, satisfy the fatness condition.
\end{enumerate}

\noindent Since symplectically fat fiber bundles are  rare and difficult to construct,  it is tempting to find some other classes of them. In this paper we do find such classes. Here are the main results.
\begin{enumerate}
\item Let now $K$ be a semisimple connected Lie group, $H\subset K$ a compact subgroup such that $H=Z_K(T)$ for some torus $T$ in $K$. Using the Kirwan map as an analogue of the moment map for the non-abelian case, we prove the following. Let there be given an associated bundle
$$M\rightarrow K\times_HM\rightarrow K/H,$$
where $(M,\omega)$ is a closed symplectic manifold with a Hamiltonian $H$-action. Then this bundle is symplectically fat (Theorem \ref{gl}).
\item Now assume that $K$ is a compact semisimple Lie group and $H\subset K$ is its subgroup. We find sufficient conditions ensuring that a $G$-structure 
$$G\rightarrow P\rightarrow K/H$$
over compact reductive homogeneous space $K/H$ admits a symplectically fat associated bundle
$$G/G_{\xi}\rightarrow P\times_{G_{\xi}}(G/G_{\xi})\rightarrow K/H$$ 
with coadjoint orbits $G/G_{\xi}$ as fibers (Theorem \ref{thm:bilinear} and Corollary \ref{cor:lerman}) in terms of the isotropy representation. This is simple but important  idea of this paper.
\item Theorem \ref{thm:bilinear} yields conditions on the isotropy representation ensuring that the twistor bundle over a homogeneous space $K/H$ of maximal rank is symplectic (Theorem \ref{thm:lerman-twistor}). Note that in this formulation we mean that $K/H$ is of maximal rank, that is $\text{rank}\,H=\text{rank}\,K$.
\item We obtain new examples of symplectic twistor bundles over non-symplectic homogeneous spaces $K/H$ (Proposition \ref{prop:Cn}).
\end{enumerate}
The constructed classes of symplectically fat fiber bundles are not covered by the previously known results. 

Our interest to twistor bundles is, in part, motivated by the result of Reznikov \cite{rez}, who proved symplecticness of twistor bundles over even-dimensional Riemannian manifolds satisfying restrictions on sectional curvature (which we have already mentioned). Moreover, in \cite{ktw}, more conceptual proof (based on symplectic fatness) was given. In this work we are interested in proving symplectic fatness of twistor bundles over non-symplectic homogeneous bases.  
We complete the introduction by mentioning a relation between symplectic fatness and the notion of the coupling form due to Sternberg.
Let $\Mo$ be a closed symplectic manifold with a Hamiltonian action
of a Lie group $G$ and a moment map $\mu:M\to \mathfrak g^*$. 
Consider the associated  bundle
$$
\Mo \to E:=P\times_G M \to B.
$$

Sternberg \cite{ster} constructed a certain closed two--form
$\Omega\in\Omega^2(E)$ associated with the connection $\theta$.  It is
called the {\em coupling form} and pulls back to the symplectic form
on each fiber and it is degenerate in general. However, if the image
of the moment map consists of fat vectors then the coupling form is
non-degenerate. In the latter case it coincides with the symplectic form given by Theorem \ref{thm:s-w}. The Sternberg's coupling form has the following properties:
\begin{enumerate}
\item $i^*\Omega=\omega$ and
\item $\pi_{!}\Omega^{n+1}=0$, $\dim\,M=2n$,
\end{enumerate}
where $i$ denotes the fiber inclusion,  $\pi: E\rightarrow B$ the fiber bundle projection, and $\pi _!$ is the fiber integration. 
These properties also serve as an abstract definition of the coupling form \cite{ktw}. The non-degeneracy of the latter can be used in studying some homotopy properties of the group of Hamiltonian diffeomorphisms. Since this is not the aim of the present work, we refer to \cite{ktw} for the details.

\section{Fatness of the canonical connection in the principal bundle $H\rightarrow G\rightarrow G/H$}\label{sec:lermantheorem}

Here we  introduce some notation which will be used throughout this work. We denote by $\mathfrak{g}$ the Lie algebra of a semisimple Lie group $G$. Let $H$ be a compact subgroup of maximal rank in $G$ and let $\mathfrak{h}$ denote the Lie algebra of $H$. The symbols $\mathfrak{g}^{\mathbb{C}}$, $\mathfrak{h}^{\mathbb{C}}$,... denote the complexifications.  Let
$\mathfrak{t}$ be a maximal abelian subalgebra in $\mathfrak{h}$. Then
$\mathfrak{t}^{\mathbb{C}}$ is a Cartan subalgebra in $\mathfrak{g}^{\mathbb{C}}$.  We denote by
$\Delta=\Delta(\mathfrak{g}^{\mathbb{C}},\mathfrak{t}^{\mathbb{C}})$ the root system of $\mathfrak{g}^{\mathbb{C}}$
with respect to $\mathfrak{t}^{\mathbb{C}}$. Under these choices
the root system for $\mathfrak{h}^{\mathbb{C}}$ is a subsystem of $\Delta$. Denote
this subsystem by $\Delta(\mathfrak{h})$.

If the Killing form $B_\mathfrak{g}$ is nondegenerate on $\mathfrak h$ then the
subspace 
$$
\mathfrak m:= 
\{X\in \mathfrak g\,|\, B_\mathfrak{g}(X,Y)=0,\, \text{for all } Y\in \mathfrak h\,\}
$$
defines a decomposition
$$
\mathfrak{g}=\mathfrak{h}\oplus\mathfrak{m}. 
$$
The decomposition is $\operatorname{Ad}_H$-invariant and the
restriction of the Killing form to $\mathfrak m$ is nondegenerate.
The decomposition complexifies to
$\mathfrak{g}^{\mathbb{C}}=\mathfrak{h}^{\mathbb{C}}\oplus\mathfrak{m}^{\mathbb{C}}$. 
Thus, we have root space decompositions:

\begin{align*}
\mathfrak{g}^{\mathbb{C}} &= \mathfrak{t}^{\mathbb{C}}+\sum_{\alpha\in\Delta}\mathfrak{g}^{\alpha},\\
\mathfrak{h}^{\mathbb{C}} &= \mathfrak{t}^{\mathbb{C}}+\sum_{\alpha\in\Delta(\mathfrak{h})}\mathfrak{g}^{\alpha},\\
\mathfrak{m}^{\mathbb{C}} &= \sum_{\alpha\in\Delta\setminus\Delta(\mathfrak{h})}\mathfrak{g}^{\alpha}.
\end{align*}

\noindent Since $G$ is semisimple, the Killing form $B_\mathfrak{g}$ defines an isomorphism
$$\mathfrak g \cong \mathfrak g^*$$ between the Lie algebra of $G$ and
its dual. The composition
$$
\mathfrak h \hookrightarrow \mathfrak g \stackrel{\cong}\longrightarrow 
\mathfrak g^* \to \mathfrak h^*
$$
is an $Ad_H$-equivariant isomorphism. Let us denote this isomorphism by 
$X_u\mapsto u$. Let $C\subset\mathfrak{t}$ be a Weyl chamber of $G$
and let $C_{\alpha}$ denote its wall determined by the root $\alpha$, that is $C_{\alpha}=\ker\alpha$ is the hyperplane defined by the root $\alpha$. 

In the formulation below we need the notion of the canonical connection in the principal bundle 
$$H\rightarrow G\rightarrow G/H.$$ The definition (in the more general setting of invariant connections in $G$-structures over reductive homogeneous spaces) is given in Section \ref{sec:g-structures} (see Remark \ref{r2}).

\begin{theorem}\label{T:lerman} 
Let $G$ be a semisimple Lie group, and $H\subset G$ a compact subgroup
of maximal rank. Suppose that the Killing form $B_\mathfrak{g}$ of $G$ is
nondegenerate on the Lie algebra $\mathfrak h\subset \mathfrak g$ of
the subgroup $H.$ Let $u\in\mathfrak{h}^*$. The following conditions are equivalent
\begin{enumerate}
\item
A vector $u\in\mathfrak{h}^*$ is fat with respect to the
the canonical invariant connection in the principal bundle
$$
H\rightarrow G\rightarrow G/H.
$$ 
\item
The vector $X_u$ does not belong to the set
$$
Ad_H(\cup_{\alpha\in\Delta\setminus\Delta(\mathfrak{h})}C_{\alpha}).
$$
\item
The isotropy subgroup $V\subset H$ of $u\in \mathfrak h^*$ with respect to the coadjoint
action is the centralizer of a torus in $G$.
\end{enumerate}
\end{theorem}

\noindent This theorem is a generalization of a theorem of Lerman \cite{ler}, and is proved in \cite{ktw}. Note that the cited result yields conditions on the fatness of vector $u$. Assume now that we are given a symplectic manifold $(M,\omega)$ with a Hamiltonian action of a Lie group $H$ and a moment map $\mu$. Assume that the image $\mu(M)$ consists of vectors satisfying Condition 2 of Theorem \ref{T:lerman}. Then, by Theorem \ref{thm:s-w}, the associated bundle 
$$M\rightarrow G\times_HM\rightarrow G/H$$
is symplectically fat. In general, there is no way in sight to check Condition 2 of Theorem \ref{T:lerman} for non-homogeneous fibers. However, if $H$ is abelian, then, by Delzant theorem \cite{a}, \cite{d}, for any Delzant polytope there exists a toric $H$-manifold $M$ whose moment map has the image which is exactly the given Delzant polytope. As a result, there exist symplectically fat fiber bundles with abelian structure group and with fibers which are toric manifolds (we take the  Delzant polytopes which omit the "forbidden" walls $C_{\alpha}$).

Thus, one of the aims of this work is to  find analogues of the described construction in the non-abelian case, that is, to find symplectically fat fiber bundles with non-homogeneous fibers and with non-abelian structure groups.

\section{Kirwan map and its properties}

In this section, following \cite{gs} we use a different notation of the Weyl chambers. Let $K$ be a compact and connected Lie group with the Lie algebra $\mathfrak{k}$ and denote by $T$ a maximal torus in $K$ (with the Lie algebra $\mathfrak{t} \subset \mathfrak{k}$). Let $W_{K}=N(T)/T$ be the Weyl group of $K.$ The group $W_{K}$ acts on $\mathfrak{t}$ and on its dual $\mathfrak{t}^{\ast}$ (there is a natural imbedding $\mathfrak{t}^{\ast}=(\mathfrak{k}^{\ast})^{T}\hookrightarrow \mathfrak{k}^{\ast}$). Every adjoint orbit in $\mathfrak{k}$ intersects $\mathfrak{t}$ in a single $W_{K}$ orbit, therefore
$$\mathfrak{k}^{\ast}/K = \mathfrak{t}^{\ast}/W_{K}$$
Choose any connected component of the set
$$\mathfrak{t}^{\ast}_{reg}= \{  \xi \in \mathfrak{t}^{\ast} \ | \ K_{\xi}=T \}$$
and denote its closure by $\mathfrak{t}^{\ast}_{+}.$ Here we denote by $K_{\xi}$ the isotropy subgroup of $\xi$ under the coadjoint action.
The set $\mathfrak{t}^{\ast}_{+}$ is called a closed Weyl chamber of $\mathfrak{k}^*.$
The closed Weyl chamber is a fundamental domain of the coadjoint action of $K$ on $\mathfrak{k}^{\ast}.$ We therefore have
$$\mathfrak{k}^{\ast}/K = \mathfrak{t}^{\ast}/W_{K}=\mathfrak{t}^{\ast}_{+}.$$

Now assume that $(M, \omega)$ is a compact connected symplectic manifold and $K \times M \rightarrow M$ is a Hamiltonian action with a moment map $\mu : M\rightarrow \mathfrak{k}^{\ast}.$ Composing $\mu$ with the natural projection $\mathfrak{k}^{\ast} \rightarrow \mathfrak{k}^{\ast}/K=\mathfrak{t}^{\ast}_{+},$ we obtain a map
$$\Phi :M \rightarrow \mathfrak{t}^{\ast}_{+},$$
which  will be called  {\em the Kirwan map}.  
The following theorem is valid for the image of it.

\begin{theorem}[Kirwan \cite{ki}]
The set $\Phi (M)$ is a convex polyhedron.
\label{kr}
\end{theorem}

\noindent The moment map $\mu :M\to \mathfrak{k}^\ast$ is equivariant with respect to the given action of $K$ on $M$ and the coadjoint action of $K$ on $\mathfrak{k}^\ast$. Therefore the image of $M$ under $\mu $ is a union of coadjoint orbits in $\mathfrak{k}^\ast$. Moreover, the set $\mu (M)$ may be entirely recovered from the image $\Phi (M)$ of the Kirwan map. More precisely, 
$$\mu(M) = K\cdot \Phi (M)$$
where $K\cdot \Phi (M)$ stands for the union of $K$-orbits of  elements of $\Phi (M)$. 
It is also clear from the discussion above that 
$$\Phi (M) = \mu(M)\cap \mathfrak{t}^\ast _+.$$

Having in mind the equalities above, we make the following observation.

\begin{theorem}
Let $(M,\omega )$ be a symplectic manifold with a Hamiltonian action of a Lie group $K$ with a moment map $\mu : M\to \mathfrak{k}^\ast$. Let $T$ be a maximal torus in $K$ and $\mu _T:M\to \mathfrak {t}^\ast$ its moment map for the restricted action of $T$ on $M$. If the set $\mu _T(M)$ consists of fat vectors with respect to some connection in the principal bundle of the form $K\to P\to B$, then $\mu (M)$ consists of fat vectors with respect to this connection as well.
\end{theorem}
\begin{proof}
The moment maps $\mu_T$ and $\mu$ are related by the equality
$$\mu _T =r\circ \mu$$
in which $r$ stands for the restriction to $\mathfrak{t}$ and one has the following inclusion $$\mu (M)\cap \mathfrak{t}^\ast _+\subset \mu _T(M).$$ So if the set $\mu_T(M)$ consists of fat vectors then $\mu(M)\cap \mathfrak{t}^\ast _+$ and $K\cdot (\mu(M)\cap \mathfrak{t}^\ast _+)$ consist of fat vectors as well.   
\end{proof}

\section{Fat bundles with arbitrary hamiltonian fibers}

Let $K$ be a semisimple connected Lie group (with the Lie algebra $\mathfrak{k}$) and $H\subset K$ a connected and compact subgroup of maximal rank (with the Lie algebra $\mathfrak{h}$) such that the Killing form  of $\mathfrak{k}$ is non-degenerate on $\mathfrak{h}.$ Choose a Cartan subalgebra $\mathfrak{t}$ of $\mathfrak{h}.$ Let $\Delta$ be the root system for $\mathfrak{k}^{\mathbb{C}}.$ Let $C_{\alpha}$ denote the hyperplane determined by the root $\alpha \in \Delta$ in the Cartan subalgebra of $\mathfrak{k}.$ Recall that one can show that in this case a root system for $\mathfrak{h}^{\mathbb{C}}$ is a subsystem of $\Delta$ (we will denote it by $\Delta(\mathfrak{h})$).

Now we are ready to state the main result of this section. 

\begin{theorem}\label{gl} Let $H$ be the centralizer of a torus in $K$. Assume that $H$ is compact and acts in a Hamiltonian fashion with the moment map $\mu$ on a compact and connected symplectic manifold $(M,\omega).$ Then the associated bundle
$$M\rightarrow K \times_{H} M \rightarrow K/H$$
is symplectically fat.
\end{theorem}

\noindent Before providing a proof of Theorem \ref{gl}, we need to present the following  lemma. 

\begin{lemma}\label{lem}
Let $V$ be a finite-dimensional Euclidean space, $W_{1}, W_{2}, ... ,W_{k}$ subspaces of $V,$ and $K\subset V$ a compact subset. Then for every nonzero vector $v\in V$ such that $v\notin W_{i}$ for $1\leq i\leq k$ there exists a scalar $\lambda$ such that
$$K+\lambda v \subset V \backslash (\bigcup_{i=1}^{k}W_{i}).$$
\end{lemma}
\begin{proof}
We shall prove this lemma by contradiction. Choose $v\in V$ so that $v\notin W_{i}$ for $1\leq i\leq k$ and assume there are sequences $(\hat{\lambda}_{n}) \subset \mathbb{R},$  $\lim_{n \to \infty}  \hat{\lambda}_{n} =+\infty$ and $(\hat{k}_{n}) \subset K$ such that
$$\forall_{j\in \mathbb{N}} \ \exists_{1\leq i \leq c} \ \hat{k}_{j}+\hat{\lambda}_{j} v\in W_{i}.$$
Therefore there exist subsequences $(\lambda_{t}),$ $(k_{t})$ and a number $s, \ 1\leq s \leq k$ so that
$$\forall_{t\in \mathbb{N}} \ k_{t}+\lambda_{t} v \in W_{s}$$ 
thus
$$\forall_{t\in \mathbb{N}} \ \exists_{w_{t}\in W_{s}} \ k_{t}+\lambda_{t} v = w_{t}.$$
Let $p_{W_{s}}$ be an orthogonal projection on $W_{s}$ and define for $u\in V$
$$\textrm{dis}(u,W_{s}):=\inf_{w\in W_{s}} \| u-w \|=\|  u-p_{W_{s}}(u) \|.$$
Since $v\notin W_{s}$ and $\lim_{t \to \infty} \| \lambda_{t}v \|=+\infty$ we have
$$\lim_{t\to \infty} \ \textrm{dis}(\lambda_{t}v, W_{s})=+\infty .$$
Because
$$\| k_{t} \|=\| w_{t}-\lambda_{t}v \| \geq \textrm{dis}(\lambda_{t}v, W_{s})$$
we obtain
$$\lim_{t\to \infty} \|  k_{t} \|=+\infty.$$
But $K$ is compact, a contradiction.
\end{proof}
\noindent Now we turn to the proof of Theorem \ref{gl}.
\begin{proof}
Let $\mathfrak{z}$ denote the center of $\mathfrak{h}$ and let $\Pi \subset \Delta$ be a system of simple roots. Since $H$ is compact we see (following the proof of Theorem 1.3, Chapter 6 in \cite{ov})  that the Lie algebra $\mathfrak{h}$ is of the form $\mathfrak{h} = \mathfrak{z}_{\mathfrak{k}}(x_{\Sigma})$ where
$$\alpha(x_{\Sigma})=\begin{cases}
0, & \mathrm{if} \quad \alpha\in \Sigma\cr
1, & \mathrm{if} \quad \alpha\in\Pi \backslash \Sigma\cr\end{cases}
$$
and $\Sigma$ denotes some subset of $\Pi.$ One can easily verify that
\begin{equation}
\Delta(\mathfrak{h})=\{\alpha\in\Delta\,|\,\alpha(x_{\Sigma})=0\},
\label{r1}
\end{equation}
and $\Sigma$ is a subset of simple roots for $\Delta (\mathfrak{h})$.
Denote by $\tilde{\mathfrak{t}} \subset \mathfrak{t}$ a Cartan subalgebra of the semisimple component $[\mathfrak{h}, \mathfrak{h}]$ of $\mathfrak{h}.$ We have
$$\mathfrak{t} =\mathfrak{z}+\tilde{\mathfrak{t}},$$
and
$$\mathfrak{h}^{\mathbb{C}} = \mathfrak{z}^{\mathbb{C}} + \tilde{\mathfrak{t}}^{\mathbb{C}} + \sum_{\alpha\in\Delta(\mathfrak{h})}\mathfrak{k}_{\alpha}.$$

Recall that if $X_{u} \in \mathfrak{h}$ then $u \in \mathfrak{h}^{\ast}$ denotes its dual (as in Section \ref{sec:lermantheorem}).
Take any nonzero $Z \in \mathfrak{z}$. Let $X_{\xi}:=Z \in \mathfrak{h}$. Then $\xi \in \mathfrak{h}^{\ast}$ is a fixed point of the coadjoint action of $H$.  
It follows that the map $\tilde{\mu}: M\rightarrow \mathfrak{h}^*$ defined by the formula
$$\tilde{\mu}=\xi+\mu$$
is a moment map as well, because $\xi$ is a constant vector invariant with respect to the coadjoint action of $H$, and, therefore, $\tilde{\mu}$ is $H$-equivariant.
Consider the Kirwan map $\tilde{\Phi}: M\rightarrow \mathfrak{t}^*_+$, where $\mathfrak{t}^*_+$ is the fundamental Weyl chamber of $\mathfrak{h}^*$ determined by this new moment map $\tilde{\mu}$.
Taking into consideration Theorem \ref{kr}, we see that
$$\tilde{\Phi}( M)=\xi + \Phi (M)$$
is a convex polyhedron. Thus we can translate $\Phi (M)$ by an arbitrary vector $\xi$ belonging to the annihilator of $[\mathfrak{h},\mathfrak{h}]$ (that is to say, we can translate $\Phi (M)$ in the ``abelian direction'' given by the center of $H$). 

Note that since $Z\in\mathfrak{z}$, $\alpha(Z)=0$ for any $\alpha\in\Delta(\mathfrak{h})$. 
However, taking into consideration Equation (\ref{r1}), we have $\beta(x_{\Sigma})\not=0$ for any $\beta\in \Delta\setminus\Delta(\mathfrak{h})$. This means that $x_{\Sigma}\not\in C_{\beta}, \beta\in \Delta\setminus\Delta(\mathfrak{h})$. Thus, for  a ``long enough'' (with  respect to the Killing form of $\mathfrak{k}$) vector $\xi$, where $X_{\xi} := cx_{\Sigma}$ and $c \in \mathbb{R}_{+}$, $\tilde{\Phi}(M)$ consists of fat vectors with respect to the canonical connection in the principal bundle $H \rightarrow K \rightarrow K/H$ (see Lemma \ref{lem} and then Theorem \ref{T:lerman}).
Moreover the image \mbox{$\tilde{\mu}(M) = {\xi} + \mu(M)$} can be recovered from the set $\tilde{\Phi}(M)$ in the following way
$$\tilde{\mu}(M) = H\cdot \tilde{\Phi}(M).$$
Now since the coadjoint orbit of a fat vector is fat (that is, consists of fat vectors), it follows that the set $\tilde{\mu}(M)$ is fat as a union of fat coadjoint orbits.
Finally, it follows from  Theorem \ref{sw} that  the bundle $M\rightarrow K \times_{H} M \rightarrow K/H$ is symplectically fat.
\end{proof}

\begin{remark} By \cite{wein} (Theorem 3.3, and remarks before it) the following holds.
Let there be given a principal bundle

 $$G\rightarrow P\rightarrow B$$
  over a {\it symplectic} base $B$, and $(M,\omega)$ be a compact Hamiltonian $G$-manifold with the moment map $\mu$. Then, one can use Thurston's theorem to obtain a fiberwise symplectic structure on the total space of the associated bundle
$$M\rightarrow P\times_GM\rightarrow B.$$
Theorem \ref{gl} shows that if the base is a homogeneous symplectic space $K/H$ of a semisimple compact Lie group $K$, then one  can also construct a fiberwise symplectic structure using Theorem \ref{thm:s-w}, because in this case $\mu(M)$ is always fat with respect to the canonical connection in the principal $H$-bundle $K\rightarrow K/H$. 
Note that such $K/H$ always satisfies the assumptions of Theorem \ref{gl} (see \cite{to}, Chapter 5).
\end{remark}

\section{Symplecticness of bundles associated with $G$-structures over homogeneous spaces}\label{sec:g-structures}

In this Section we will use the theory of invariant connections on homogeneous spaces in the form presented in Sections 1 and 2 of Chapter X of \cite{kn2}. Let $M$ be a smooth manifold of dimension $n$, and let
$$G\rightarrow P\rightarrow M $$
be a $G$-structure, that is, a reduction of the frame bundle $L(M)\rightarrow M$ to a Lie group $G$. Any diffeomorphism $f\in \operatorname{Diff}(M)$ acts on $L(M)$ by the formula 
$$f(u):=(df_xX_1,...,df_xX_n)$$
for any frame $u=(X_1,...,X_n), X_i\in T_xM$ over a point $x\in M$. By definition, $f$ is called an automorphism of the given $G$-structure, if this action commutes with the action of $G$.

Let $M=K/H$ be a homogeneous space of a connected Lie group $K$.  Assume that $M$ is equipped with a $K$-invariant $G$-structure. The latter means that any left translation $\tau(k): K/H\rightarrow K/H$, $\tau(k)(aH)=kaH$ lifts to an automorphism. Let $o=H\in K/H$. Consider the linear isotropy representation of $H,$ that is, a homomorphism $H\rightarrow GL(T_oM)$ given by the formula
$$h \mapsto d\tau(h)_o,\,\text{for}\,h\in H,o=H\in K/H.$$ 
It is important to observe that we can fix a frame $u_o:\mathbb{R}^n\rightarrow T_oM,$ $u_{o}\in P$ and identify the linear isotropy representation of $H$ with a homomorphism $\lambda :H \rightarrow G$
$$\lambda(h)=u_o^{-1}d\tau(h)_ou_o,\,h\in H.$$
One can see this as follows. Denote by $P_{o}\subset P$ the fiber over the point $o,$ then $u_{o} \in P_{o}$.  The action of $H$ lifted to $P$ preserves $P_{o}$, hence  $h(u_{o}) \in P_{o}.$ Since the structure group $G$ acts transitively on $P_{o}$, there exists exactly one $g\in G$ such that
$$h(u_{o})=u_{o}g.$$
It is easy to see that $\lambda (h)=g.$ 

In the sequel we assume that $K/H$ is reductive. In this case $\mathfrak{k}$ can be decomposed into a direct sum
$$\mathfrak{k}=\mathfrak{h}\oplus\mathfrak{m}$$
such that $Ad_{H}(\mathfrak{m})\subset\mathfrak{m}$. Note that the latter implies $[\mathfrak{h},\mathfrak{m}]\subset\mathfrak{m}$. Also, we assume that the isotropy representation is faithful. Let us make one more straightforward but important observation.  One can  identify $\lambda$ with the restriction of the adjoint representation of $H$ on $\mathfrak{m}$  (which we also denote by $\lambda$). 

Note that the identification of the isotropy representation with the restriction of the adjoint representation on $\mathfrak{m}$ is used in \cite{kn2}, Chapter X, for example, in the proof of Theorem 2.6.

We say that a connection $\theta$ in $P\rightarrow M$ is $K$-invariant, if for any $k\in K$ the lift of $\tau(k)$ preserves it. We need the following description of the set of invariant connections in the principal bundle $P\rightarrow M$ from \cite{kn2}, Chapter X.

\begin{theorem} Let there be given a $K$-invariant $G$-structure over a reductive homogeneous space $M=K/H$. There is a one-to-one correspondence between the $K$-invariant connections in it, and $Ad_H$-invariant linear maps
$$\Lambda_\mathfrak{m}: \mathfrak{m}\rightarrow \mathfrak{g}.$$
\end{theorem}

\noindent Note that here $Ad_H$-invariance means that
$$\Lambda_\mathfrak{m}(Ad\,h(Z))=\lambda(h)(\Lambda_\mathfrak{m}(Z)),\,Z\in\mathfrak{m},h\in H.$$
\begin{remark} {\rm Recall that a connection in the given $K$-invariant $G$-structure is called {\it canonical} if it corresponds to the map $\Lambda_{\mathfrak{m}}=0$. }
\label{r2}
\end{remark}
The curvature of such connection is described by the following result \cite{kn1} (Theorem II.11.7).  

\begin{theorem}\label{thm:connection} The curvature form of the canonical connection in $P$ is given by the formula 
$$\Theta(X,Y)=-\lambda([X,Y]_{\mathfrak{h}}), X,Y\in\mathfrak{m}.$$
\end{theorem}

\noindent Let us introduce the following notation. Denote by $\lambda^*:\mathfrak{g}^*\rightarrow\mathfrak{h}^*$ the dual map defined by $\lambda^*(f)(H)=f(\lambda(H))$. Let $B_{\mathfrak{g}}$ and $B_{\mathfrak{h}}$ be some non-degenerate bilinear invariant forms on $\mathfrak{g}$ and $\mathfrak{h}$ (these may be the Killing forms, for example, if the corresponding Lie algebras are semisimple). We use them to define pairings between $\mathfrak{g}^*$ and $\mathfrak{g}$ and $\mathfrak{h}^*$ and $\mathfrak{h}$:
$$B_{\mathfrak{g}}(X_f,X)=\langle f,X\rangle,\,B_{\mathfrak{h}}(Y_g,Y)=\langle g,Y\rangle,$$
for $f\in\mathfrak{g}^*,X\in\mathfrak{g},g\in\mathfrak{h}^*,Y\in\mathfrak{h}$.  If $\lambda^*(f)\in\mathfrak{h}^*$, then the $B_{\mathfrak{h}}$-dual of $\lambda^*(f)$ will be denoted by $X^{\lambda}_f$, that is
$$B_{\mathfrak{h}}(X^{\lambda}_f,Y):=\langle\lambda^*(f),Y\rangle.$$
Consider now the additional assumption that we are given a homogeneous space $K/H$ which is  reductive and that $K$ is semisimple. Denote by $B_{\mathfrak{k}}$ the Killing form of $\mathfrak{k}$, and by $B_{\mathfrak{h}}$ the restriction of $B_{\mathfrak{k}}$ on $\mathfrak{h}$. Assume that $B_{\mathfrak{h}}$ is non-degenerate.

\begin{theorem}\label{thm:bilinear} Let there be given a $G$-structure over the reductive homogeneous space $K/H$ of semisimple Lie group $K$. Assume that $rank\, K=rank\, H$. Then $v\in\mathfrak{g}^*$ is fat with respect to the canonical connection, if and only if the 2-form
$$B_{\mathfrak{k}}(X^{\lambda}_v,[X,Y]),\,X,Y\in\mathfrak{m}$$
is non-degenerate.
\end{theorem}
\begin{proof} By definition $v\in\mathfrak{g}^*$ is fat with respect to the canonical connection, if and only if the 2-form
$$\langle v,\Theta(X,Y)\rangle=\langle v,-\lambda([X,Y]_{\mathfrak{h}})\rangle$$
is non-degenerate. Moreover
$$\langle v,\lambda([X,Y]_{\mathfrak{h}})\rangle = \langle \lambda^* v,[X,Y]_{\mathfrak{h}}\rangle=$$
$$B_{\mathfrak{h}}(X^{\lambda}_v,[X,Y]_{\mathfrak{h}})=B_{\mathfrak{k}}(X_v^{\lambda},[X,Y]_{\mathfrak{h}})=B_{\mathfrak{k}}(X^{\lambda}_v,[X,Y]).$$

\noindent
Note that the last equality follows because $\mathfrak{h}$ and $\mathfrak{m}$ are $B_{\mathfrak{k}}$-orthogonal.
\end{proof}

\begin{corollary}\label{cor:lerman}  Under the assumptions of  Theorem \ref{thm:bilinear}, $v\in\mathfrak{g}^*$ is fat, if 
$$X_v^{\lambda}\not\in Ad(H)(\cup_{\alpha\in\Delta\setminus\Delta(\mathfrak{h})}C_{\alpha}).$$
\end{corollary}
\begin{proof} The proof follows from Theorem \ref{T:lerman}, since it  uses the same identification of $\mathfrak{h}^*$ and $\mathfrak{h}$ via the non-degenerate Killing form.
\end{proof}

\noindent Denote by $J$ the matrix in $\mathfrak{s}\mathfrak{o}(2n)$ consisting of $n$ blocks of the form
$$J=\left( \begin{array}{cc}
0 & 1\\
-1 & 0
\end{array}
\right).$$
It is known that the homogeneous space $SO(2n)/U(n)$ is symplectic, because it is the coadjoint orbit of the dual vector $J^{*}$ (with respect to the Killing form $B_{\mathfrak{so}(2n)}$ and the standard transitive action of $SO(2n)$ on $SO(2n)/U(n)$) (section 3.4 of \cite{gls}).

Recall that the twistor bundle over an even-dimensional Riemannian manifold $(M^{2n},g)$ is the bundle of complex structures in the tangent spaces $T_{p}M$. More precisely, it is the bundle associated with the orthonormal frame bundle of $M$ with the fiber $SO(2n)/U(n)$:

$$SO(2n)/U(n)\rightarrow SO(K/H)\times_{SO(2n)}(SO(2n)/U(n))\rightarrow K/H$$
where $\dim\,K/H=2n$, and $SO(K/H)$ stands for the total space of the principal $SO(n)$-bundle of oriented frames. In what follows we will denote the total space of the twistor bundle by $\mathcal{T}(K/H)$. 

We apply Corollary \ref{cor:lerman} to get a sufficient condition on fatness of the twistor bundle.

\begin{theorem}\label{thm:lerman-twistor} Let there be given an even-dimensional reductive homogeneous space $K/H$ of a semisimple Lie group $K$. Assume $\text{rank}\,K=\text{rank}\,H$ and $\dim\,K/H=2n$. 
Consider  the twistor bundle
$$SO(2n)/U(n)\rightarrow \mathcal{T}(K/H)\rightarrow K/H.$$
  Let $\lambda:\mathfrak{h}\rightarrow\mathfrak{g}= \mathfrak{s}\mathfrak{o}(2n)$ be the isotropy representation. Let $J^*\in\mathfrak{s}\mathfrak{o}(2n)^*$ denote the dual to $J$ with respect to the Killing form $B_{\mathfrak{g}}$. Assume $X_{J^*}^{\lambda}\in\mathfrak{h}$ has the  property
$$X_{J^*}^{\lambda}\not\in Ad_H(\cup_{\alpha\in\Delta\setminus\Delta(\mathfrak{h})}C_{\alpha}).$$
Then the given twistor bundle is symplectically fat. 
\end{theorem} 
\begin{proof} It follows from Corollary \ref{cor:lerman} that 
$$X_{J^*}^{\lambda}\not\in Ad_H(\cup_{\alpha\in\Delta\setminus\Delta(\mathfrak{h})}C_{\alpha})$$
implies fatness of $J^*$. As a result, the coadjoint orbit of $J^*$ is also fat. Thus, the fiber $SO(2n)/U(n)$ has fat image under the moment map (in this case the moment map is just the inclusion of this orbit into $\mathfrak{so}(2n)$). Now applying Theorem \ref{sw} we get the result.
\end{proof}

\noindent Now we are looking for some verifiable conditions when Theorem \ref{thm:lerman-twistor} is applicable. 
Recall the  observation that  the isotropy representation can be identified with the adjoint representation restricted on $\mathfrak{m}$.

\begin{theorem}\label{thm:sufficient-twistors} Let there be given a reductive homogeneous space $K/H$ of dimension $2n$ Assume that the following assumptions hold:
\begin{enumerate}
\item $K$ is compact, semisimple and the Killing form of $\mathfrak{k}$ determines the invariant Riemannian metric on $K/H$;
\item $\mathrm{rank}\;K=\mathrm{rank}\;H$;
\item there exists $T\in\mathfrak{t}\subset\mathfrak{h}$ in the Cartan subalgebra $\mathfrak{t}$ of $\mathfrak{h}$ and $\mathfrak{k}$ such that 
$$(ad\,T|_{\mathfrak{m}})^2=-\operatorname{id}_{\mathfrak{m}}, \,T\not\in \cup_{\alpha\in\Delta\setminus\Delta(\mathfrak{h})}C_{\alpha};$$
\item for $J=ad\,T|_{\mathfrak{m}}$ the isotropy representation $\lambda: \mathfrak{h}\rightarrow\mathfrak{s}\mathfrak{o}(\mathfrak{m})$ satisfies the condition
$$X_{J^*}^{\lambda}=T.$$
 \end{enumerate}
Then the twistor bundle
$$SO(2n)/U(n)\rightarrow \mathcal{T}(K/H)\rightarrow K/H$$
 over the Riemannian homogeneous space $K/H$ is symplectically fat.
\end{theorem}
\begin{proof} We see that the dual $X^{\lambda}_{J^*}$ of $\lambda^*(J^*)$ must be $T$, 
$T=X_{J^*}^{\lambda}.$
Therefore, $J^*\in\mathfrak{s}\mathfrak{o}^*(\mathfrak{m})$ is fat. By the assumption $3$, $J^2|_{\mathfrak{m}}=-\operatorname{id}_{\mathfrak{m}}$, moreover $J$ is skew-symmetric with respect to the Killing form $B_{\mathfrak{k}}$. Thus $J\in\mathfrak{s}\mathfrak{o}(\mathfrak{m})$ and represents some complex structure on vector space $\mathfrak{m}$. Therefore the coadjoint orbit  dual to the adjoint orbit of $J$ consists of fat vectors, and the proof follows, as in Theorem \ref{thm:lerman-twistor}. 
\end{proof}

\noindent The latter result yields examples of homogeneous spaces with symplectic twistor bundles over them. In order to describe these examples, recall that the compact real form of any semisimple complex Lie algebra $\mathfrak{g}^\mathbb{C}$ can be given by the formula
$$\mathfrak{g}=\sum_{\alpha\in\Delta}\mathbb{R}(iH_{\alpha})+\sum_{\alpha\in\Delta}\mathbb{R}(X_{\alpha}-X_{-\alpha})+\sum_{\alpha\in\Delta}\mathbb{R}(i(X_{\alpha}+X_{-\alpha})).$$
Here $\Delta$ denotes the root system for $\mathfrak{g}^\mathbb{C}$. 
\begin{example}\label{ex:twistors} 
Consider the compact homogeneous space $Sp(2n)/SO(2n)$. The embedding of $SO(2n)$ into $Sp(2n)$ can be described as follows (see \cite{ov}, note, however, that we do assume that the reader is familiar with the material of \cite{ov}, especially Chapters 3 and 4).  
The root system $\Delta$ of $\mathfrak{s}\mathfrak{p}(2n)$ is given by the formula.
$$\Delta = \{  \pm e_{s} \pm e_{t}, \ \pm 2e_{s} \ | \ 1\leq s,t \leq n \}$$
where $\{ e_s \}_{s=1}^{n} $ denotes the canonical orthonormal base in the standard Euclidean space $\mathbb{R}^n$. Clearly, $\Delta$ contains the subsystem
$$ \Delta (\mathfrak{h}) = \{ \pm e_{s} \pm e_{t}, \  | \ 1\leq s,t \leq n  \},$$
which determines the subalgebra $\mathfrak{h}=\mathfrak{s}\mathfrak{o}(2n)$. It is straightforward that this inclusion of Lie algebras determines the embedding $SO(2n)\rightarrow Sp(2n)$. Note that this embedding is not regular in the sense of \cite{ov}, Chapter 6, and cannot be obtained from the extended Dynkin diagram of type $C_n$.
\end{example}

\begin{proposition}\label{prop:Cn}
The twistor bundle
$$SO(2n)/U(n)\rightarrow\mathcal{T}(Sp(2n)/SO(2n))\rightarrow Sp(2n)/SO(2n)$$
over the Riemannian space $K/H=Sp(2n)/SO(2n)$ is symplectically fat. The base of this bundle is not symplectic.
\end{proposition}
\begin{proof} Begin with a straightforward remark on duality: if $\lambda: V\rightarrow W$ is a monomorphism of vector spaces endowed with non-degenerate bilinear forms $B_V$ and $B_W$ such that $B_W(\lambda(u),\lambda(v))=B_V(u,v)$, then, for the duality determined by $B_V$ and $B_W$, the following holds. If $\lambda(v)=w$, then \mbox{$\lambda^*(w^*)=v^*$}. Observe that in this example $\lambda$ defines an isomorphism between $\mathfrak{h}$ and $\mathfrak{g}$, thus we can obtain an invariant nondegenerate symmetric bilinear form $B_{\mathfrak{g}}$ induced by $B_{\mathfrak{h}}$ such that
$$B_{\mathfrak{h}}(X,Y)=B_{\mathfrak{g}}(\lambda (X), \lambda (Y)), \ X,Y \in \mathfrak{h}.$$
Now we continue with the next  general remark. Assume that we are given a compact homogeneous space $K/H$ with the property that the pair $(\mathfrak{k},\mathfrak{h})$ consists of canonical real forms of $\mathfrak{k}^\mathbb{C}$ and $\mathfrak{h}^\mathbb{C}$ (that is, the corresponding Cartan subalgebras coincide). Thus
$$\mathfrak{k}^\mathbb{C}=\mathfrak{t}^\mathbb{C}+\sum_{\alpha\in\Delta(\mathfrak{h})}\mathfrak{k}_{\alpha}+\sum_{\beta\in\Delta\setminus\Delta(\mathfrak{h})}\mathfrak{k}_{\beta}$$
$$\mathfrak{h}^\mathbb{C}=\mathfrak{t}^c+\sum_{\alpha\in\Delta(\mathfrak{h})}\mathfrak{k}_{\alpha},$$
$$\mathfrak{m}^\mathbb{C}=\sum_{\beta\in\Delta\setminus\Delta(\mathfrak{h})}\mathfrak{k}_{\beta}.$$
Here $\Delta$ again denotes the root system for $\mathfrak{k}$ and $\Delta(\mathfrak{h}) \subset \Delta$ is the corresponding root system of  $\mathfrak{h}$. Assume that we can choose $T\in\mathfrak{t}^\mathbb{C}$ satisfying the equations 
$$\alpha(T)=\pm i \ \mbox{for} \ \alpha\in\Delta\setminus\Delta(\mathfrak{h}), \ \alpha(T)\in i\mathbb{R} \ \mbox{for} \ \alpha\in\Delta(\mathfrak{h})$$
where $i$ or $-i$ are chosen in a way to ensure that the above system of linear equations has a solution.
Note that  $\mathfrak{t}^\mathbb{C}=\mathfrak{t}+i\mathfrak{t}$, where $\mathfrak{t}$ denotes the real form of $\mathfrak{t}^\mathbb{C}$ consisting of vectors $H\in\mathfrak{t}^\mathbb{C}$ such that $\alpha(H)\in\mathbb{R}$ for all $\alpha\in\Delta$. It follows that $T\in i\mathfrak{t}=\sum_{\alpha\in\Delta}\mathbb{R}(iH_{\alpha})$. Therefore, $T\in\mathfrak{k}$. But $(ad\,T|_{\mathfrak{m}})^2=-\operatorname{id}$, because by construction it satisfies this equality on $\mathfrak{m}^\mathbb{C}$. By the remark at the beginning of the proof, for $J=ad\,T|_{\mathfrak{m}}$ one obtains $X^{\lambda}_{J^*}=T$. But $T$ does not belong to any wall $C_{\alpha},\alpha\in\Delta\setminus\Delta(\mathfrak{h})$. Thus we see that under the adopted assumptions the twistor bundle over $K/H$ must be symplectically fat, by Theorem \ref{thm:sufficient-twistors}.

In our example, we can find the required $T\in\mathfrak{t}^\mathbb{C}.$ We have
$$\Delta = \{  \pm e_{s} \pm e_{t}, \ \pm 2e_{s} \ | \ 1\leq s,t \leq n \} \ \ \Delta (\mathfrak{h}) = \{ \pm e_{s} \pm e_{t}, \  | \ 1\leq s,t \leq n  \},$$
therefore
$$\Delta \backslash \Delta (\mathfrak{h}) = \{  \pm 2e_{s} \ | \ 1\leq s \leq n \}$$
Since $\{ 2e_{s}\}_{s=1}^{n}$ are linearly independent we can find $T\in \mathfrak{h}^c$ such that \linebreak $\alpha(T)=\pm i$, $\alpha\in\Delta\setminus\Delta(\mathfrak{h})$. Moreover
$$\alpha(T)\in i\mathbb{R} \  \mbox{for} \ \alpha\in\Delta(\mathfrak{h}). \eqno\qedhere$$
\end{proof}

\noindent{\bf Acknowledgment}. We express our sincere thanks to the anonymous referee for his/her valuable advice which essentially improved the contents of this article.

\end{document}